\documentclass[12pt]{amsart}
\usepackage{amsmath,amssymb}

\theoremstyle{plain}
\newtheorem{thm}{Theorem}[section]
\newtheorem{theorem}[thm]{Theorem}

\newtheorem{lemma}[thm]{Lemma}

\newtheorem{proposition}[thm]{Proposition}
\theoremstyle{definition}

\newtheorem{remark}[thm]{Remark}

\newtheorem{definition}[thm]{Definition}

\newtheorem{examples}[thm]{Examples}
\newtheorem{conjecture}[thm]{Conjecture}

\numberwithin{equation}{section}

\newcommand{\sC}{{\mathcal C}}

\newcommand{\sK}{{\mathcal K}}

\newcommand{\sU}{{\mathcal U}}

\newcommand{\C}{{\mathbb C}}

\newcommand{\BP}{{\mathbb P}}
\newcommand{\pit}{{\mathbb P}}
\newcommand{\p}{{\mathbb P}}
\newcommand{\Q}{{\mathbb Q}}

\newcommand{\BS}{{\mathbb S}}

\newcommand{\fg}{{\mathfrak g}}

\newcommand{\aut}{{\mathfrak a}{\mathfrak u}{\mathfrak t}}

\newcommand\sd{\!>\!\!\!\! \lhd \:}

\def\Gr{\mathop{\rm Gr}\nolimits}

\def\Lag{\mathop{\rm Lag}\nolimits}
\def\Sym{\mathop{\rm Sym}\nolimits}

\def\Hom{\mathop{\rm Hom}\nolimits}

\def\Ker{\mathop{\rm Ker}\nolimits}

\def\SYM{\mathop{\rm Sym}\nolimits}

\title{On Fano manifolds of Picard number one with big automorphism groups}
\author{Baohua Fu, Wenhao Ou and Junyi Xie}

\begin{document}

\begin{abstract}
Let $X$ be an $n$-dimensional smooth Fano complex variety  of Picard number one. Assume that the VMRT at a general point of $X$ is smooth irreducible and non-degenerate (which holds if $X$ is covered by lines with index $ >(n+2)/2$).  It is proven that $\dim \aut(X) > n(n+1)/2$ if and only if $X$ is isomorphic to  $\BP^n, \Q^n$ or ${\rm Gr}(2,5)$. Furthermore, the equality $\dim \aut(X) = n(n+1)/2$ holds only when $X$ is isomorphic to the 6-dimensional Lagrangian Grassmannian ${\rm Lag}(6)$ or a general hyperplane section of ${\rm Gr}(2,5)$.
\end{abstract}

\maketitle

\section{Introduction}

For a smooth projective complex variety $X$, the  Lie algebra $\aut(X)$ of its automorphism group ${\rm Aut}(X)$ is naturally identified with $H^0(X, T_X)$.
A natural question is how big can this group be. In general, $\aut(X)$ can be very big with respect to its dimension. For example, when $X$ is the Hirzebruch surface $\mathbb{F}_m$, then $\dim \aut(X)=m+5$. On the other hand, in the case of Fano manifold of Picard number one, we have the following:
\begin{conjecture}[\cite{HM}, Conjecture 2]
Let $X$ be an $n$-dimensional Fano manifold of Picard number one. Then $\dim \aut(X) \leq n^2+2n$, with equality if and only if $X \simeq \BP^n$.
\end{conjecture}

In \cite{HM} (Theorem 1.3.2), this conjecture is proven under the assumption that the variety of minimal rational tangents (VMRT for short, cf. Definition \ref{d.VMRT}) at a general point of $X$ is smooth irreducible non-degenerate and linearly normal. The purpose of this note is to push further the ideas of \cite{HM}, combined with the recent results in \cite{FH1} and \cite{FH2}, to prove the following

\begin{theorem} \label{t.main}
Let $X$ be an $n$-dimensional Fano manifold of Picard number one. Assume that the VMRT at a general point of $X$ is smooth irreducible and  non-degenerate.
Then we have
\begin{itemize}
\item[(a)]  $\dim \aut(X) > n(n+1)/2$ if and only if $X$ is isomorphic to $\BP^n, \Q^n$ or ${\rm Gr}(2,5)$.
\item[(b)] The equality $\dim \aut(X) = n(n+1)/2$ holds only when $X$ is isomorphic to ${\rm Lag}(6)$ or a general hyperplane section of ${\rm Gr}(2,5)$.
\end{itemize}
\end{theorem}
\begin{remark}
As proved in Corollary 1.3.3 \cite{HM}, the assumption on the VMRT of $X$ is satisfied if there exists an embedding $X \subset \BP^N$ such that $X$ is covered by lines with index $ > \frac{n+2}{2}$.
\end{remark}
Recall that $\dim \aut(\BP^n) = n^2+2n$ and $\dim \aut(\Q^n) = \dim \mathfrak{so}_{n+2} = \frac{(n+1)(n+2)}{2}$. The previous theorem indicates  that there may exist big gaps between the dimensions of automorphism groups of Fano manifolds of Picard number one.

To prove Theorem \ref{t.main}, we first show the following result, which could be of independent interest.

\begin{thm}
\label{thm:auto}
Let $n \geq 2$ be an integer.
 Let $X\subsetneq \p^n$ be an irreducible and non-degenerate subvariety of codimension $c \geq 1$, which is not a cone.
Let $G_n^X = \{g \in {\rm PGL}_{n+1}(\mathbb{C})| g(X) =X\}$. 
 Then
\begin{itemize}
\item[(a)]
$\dim G_n^X \leq \frac{n(n+1)}{2} - \frac{(c-1)(c+4)}{2}. $
\item[(b)]  $\dim G_n^X = \frac{n(n+1)}{2}$  if and only if $X$ is a smooth quadratic hypersurface.
\item[(c)] if $X$ is smooth and is not a quadratic hypersurface, then $\dim G_n^X \leq \frac{n(n+1)}{2} -3.$
\end{itemize}
\end{thm}

The idea of the proof of Theorem \ref{t.main} is similar to that in \cite{HM}: the dimension of $\aut(X)$ is controlled by $n + \dim \mathfrak{aut}(\hat{\mathcal{C}_x})+\dim \mathfrak{aut}(\hat{\mathcal{C}_x})^{(1)}$, where $\mathcal{C}_x$ is the VMRT of $X$ at a general point, $\mathfrak{aut}(\hat{\mathcal{C}}_x)$ is the Lie algebra of infinitesimal automorphisms of $\hat{\mathcal{C}_x}$ while $\mathfrak{aut}(\hat{\mathcal{C}_x})^{(1)}$ is the first prolongation of  $\mathfrak{aut}(\hat{\mathcal{C}}_x)$ (cf. Definition \ref{d.prolong}).  By Theorem \ref{thm:auto}, we have an optimal bound for
$\dim \mathfrak{aut}(\hat{\mathcal{C}_x})$, which gives the bound for $\dim \aut(X)$ in the case when $\mathfrak{aut}(\hat{\mathcal{C}_x})^{(1)} =0$.
For the case when  $\mathfrak{aut}(\hat{\mathcal{C}_x})^{(1)} \neq 0$, we have a complete classification of all such embeddings $\mathcal{C}_x \subset \BP T_xX$ by \cite{FH1} and \cite{FH2}. Then a case-by-case check gives us the bound in Theorem \ref{t.main}.  Finally we apply Cartan-Fubini extension theorem of Hwang-Mok (\cite{HM0}) and the result of Mok (\cite{M}) to recover the variety $X$ from its VMRT.
\bigskip

{\em Convention:} For a projective variety $X$, we denote by $\aut(X)$ its Lie algebra of automorphism group, while for an embedded variety $S \subset \BP V$, we denote by $\aut(\hat{S})$ the  Lie algebra of infinitesimal automorphisms of $\hat{S}$, which is given by
$$
\aut(\hat{S}):=\{g \in {\rm End}(V)| g(\alpha)\in T_\alpha(\hat{S}), \text{for any smooth point}\ \alpha \in \hat{S}\}.
$$

{\em Acknowledgements:}  We are very grateful to Francesco Russo and Fyodor Zak for discussions on Lemma \ref{lem:nondegenerate:cutout}, the proof of which presented here is due to Fyodor Zak.  Baohua Fu is supported by National Natural Science Foundation of China (No. 11688101, 11621061 and 11771425).

\section{Automorphism group of embedded varieties}
 For each positive integer $n$, we let $$G_n= \mathrm{Aut}(\p^n) = \mathrm{PGL}_{n+1}(\mathbb{C}).$$ If $X\subseteq \p^n$ is a subvariety, we denote by $G_n^X\subseteq G_n$ the subgroup of elements $g$  such that  $g(X)=X$.
Note that if $X \subset \p^n$ is non-degenerate, then $ G_n^X \subset {\rm Aut}(X)$.

  The goal of this section is to prove the following theorem, which is more general than Theorem \ref{thm:auto}.


\begin{thm}
\label{thm:autocone}
Assume that $n\geq 2$. Let $X\subsetneq \p^n$ be an irreducible and non-degenerate subvariety of codimension $c \geq 1$.
Then the set $$C_X:=\{x\in X|\,\, X \text{ is a cone with vertex } x\}$$ is a linear subspace.  Set $r_X:=-1$ if $C_X=\emptyset$ and $r_X:=\dim C_X$ otherwise.
Then we have
$$\dim G_n^X \leq \frac{(n-r_X-1)(n-r_X)}{2} - \frac{(c-1)(c+4)}{2}+(r_X+1)(n+1). $$
\end{thm}

The idea of the proof is  to cut $X$ by a general hyperplane, and then use induction  on $n$ to conclude.  To this end, we will first introduce the following notation.
For a hyperplane $H$ in $\p^n$, we may choose a coordinates system $[x_0:x_1:\cdots : x_n]$ such that $H$ is defined by $x_0=0$.  For every  $g\in G_n$, it has a representative  $M\in \mathrm{GL}_{n+1}(\mathbb{C})$, such that  its action on $\p^n$ is given by $g([x_0:x_1:\cdots : x_n])=[y_0:y_1:\cdots : y_n]$, where
$$\left( \begin{matrix} y_0\\ y_1\\ \vdots \\ y_n \end{matrix}  \right)  = M \left( \begin{matrix} x_0\\ x_1\\ \vdots \\ x_n \end{matrix}  \right)$$
Then  $g\in G_n^H$ if and only if it  can be represented by a matrix of the shape

$$\left(
\begin{array}{c|c}
a_0&\begin{matrix}0 & \cdots & 0 \end{matrix}   \\
\hline
 \begin{matrix}a_1\\ a_2\\ \vdots \\ a_n \end{matrix} & {A}
\end{array}
\right)
$$

There is a natural morphism   $r_H: G_n^H \to \mathrm{Aut}(H) \cong  \mathrm{PGL}_{n}(\mathbb{C})$.  Then an element $g$ is in the kernel $\mathrm{Ker}\, r_H$  if and only if it can be represented by a matrix of the shape

$$\left(
\begin{array}{c|c}
\lambda &\begin{matrix}0 & \cdots & 0 \end{matrix}   \\
\hline
 \begin{matrix}a_1\\ a_2\\ \vdots \\ a_n \end{matrix} &  {\mathrm{Id}}_n
\end{array}
\right)
$$

For such $g\in \mathrm{Ker}\, r_H$,  we call $\lambda$ the special eigenvalue of $g$. The action of $g$ on the normal bundle of $H$ is then  the multiplication by $\lambda$.
 We see that this is  independant of the choice of representatives of $g$ in  $\mathrm{GL}_{n+1}(\mathbb{C})$.   We also note that if $g,h$ are two elements in $\mathrm{Ker}\, r_H$, with special eigenvalues $\lambda$ and $\mu$ respectively, then the special eigenvalue of $gh$ is equal to $\lambda \mu$.
 This gives a  homomorphism  $\chi_H: \mathrm{Ker}\, r_H \to \mathbb{C}^*$.

Before giving the proof of Theorem \ref{thm:auto}, we will first prove several lemmas.

\begin{lemma}
\label{lem:fixing:hyperplane:non:trivial}
Let $H$ be a hyperplane in $\p^n$, and let $X\subseteq \p^n$ be any subvariety.  Then  $$\dim\, G_n^X \leq \dim\, (G_n^H\cap G_n^X) + n.$$
\end{lemma}

\begin{proof}
This lemma follows from the fact that  $\dim\, G_n = \dim\, G_n^H + n$.
\end{proof}

We also need the following Bertini type lemma.

\begin{lemma}
\label{lem:nondegenerate:cutout}
Assume that $n\geq 2$. Let $X\subseteq \p^n$ be a non-degenerate irreducible subvariety of positive dimension which is not a cone. Then for a general hyperplane $H$, the intersection $X\cap H$ is still non-degenerate in $H$ and is not a cone.
\end{lemma}

\begin{proof}
Since $X\subseteq \p^n$ is irreducible and non-degenerate, the intersection of $X$ and a general hyperplane $H$  is non-degenerate in $H$.

Let $V\subset X\times (\p^n)^*$ be the subset of pair $(x,H)$ such that $H\cap X$ is a cone with vertex $x$.
Set $\pi_1:V\to X$ and $\pi_2:X\to (\p^n)^*$ the projections to the first and the second factors. If $\pi_2$ is not surjective, we concludes the proof.
So we may assume that $\pi_2$ is surjective.
Hence $\dim V\geq n$. Set $Y:=\pi_1(V).$

We first  assume that there is  $x\in Y$ such that $\dim \pi_1^{-1}(x)\geq 1$.
Since any non trivial complete one-dimensional family of hyperplanes in $\p^n$ covers the whole $\p^n$, this condition  implies that for every point $x'\in X\backslash \{x\}$, there is some hyperplane $H$ containing $x$ and $x'$ such that $H\cap X$ is a cone with vertex $x$.
Therefore, the line joining $x$ and $x'$ is contained in $H\cap X$ and hence in $X$.
This shows that $X$ is a cone with vertex $x.$
We obtain a contradiction.

So the morphism $\pi_1:V\to Y$ is finite.
Then we get $$n\leq \dim V=\dim Y\leq \dim X<n,$$  which is a contradiction. This concludes the proof.
\end{proof}

In the following lemmas, we will show that the kernel of $G_n^X\cap G_n^H\to G_{n-1}^{X\cap H}$ is a finite set. Note that this kernel is nothing but $G_n^X\cap \mathrm{Ker}\, r_H$, as $X \cap H$ is non-degenerate by Lemma \ref{lem:nondegenerate:cutout}.  We will discuss according to the special eigenvalue of an element inside. We will first study the case when it is equal to $1$.

\begin{lemma}
\label{lem:restriction:kernel:eigenvalue:1:identity}
Assume that $n\geq 2$. Let $X\subseteq \p^n$ be an irreducible subvariety which is not a cone. Let $H$ be a hyperplane in $\p^n$ such that $X\not\subseteq H$. Let $g\in \mathrm{Ker}\, r_H \cap G_n^X$. If the special eigenvalue of $g$ is $1$, then $g$ is the identity in $G_n$. In other words, the map
$ \chi_H: \mathrm{Ker}\, r_H  \cap G_n^X \to \mathbb{C}^*$ is injective.
\end{lemma}

\begin{proof}
Assume the opposite.   We  choose  a homogeneous coordinates system $[x_0:x_1:\cdots : x_n]$ such that $H$ is defined by $x_0=0$, and that   $g([x_0:x_1:\cdots : x_n]) = [y_0:y_1:\cdots : y_n]$, where

$$ \left(  \begin{matrix}y_0\\ y_1\\ y_2\\ \vdots \\ y_n \end{matrix} \right) =
\left(
\begin{array}{c|c}
1 &\begin{matrix}0 & \cdots & 0 \end{matrix}   \\
\hline
 \begin{matrix}a_1\\ a_2\\ \vdots \\ a_n \end{matrix} &  {\mathrm{Id}}_n
\end{array}
\right)
\left(  \begin{matrix} x_0\\ x_1\\ x_2\\ \vdots \\ x_n \end{matrix} \right) =
\left(  \begin{matrix} x_0\\ x_1+a_1x_0\\ x_2+a_2x_0\\ \vdots \\ x_n +a_nx_0\end{matrix} \right)
$$

By assumption, the $a_i$ are not all equal to zero. Let $p\in \p^n$ be the point with homogeneous coordinates $[0:a_1:\cdots : a_n]$.  For any point $x \in X\backslash H$ with homogeneous coordinates $[1:x_1:\cdots :x_n]$,   the point  $g^k(x)$ has coordinates $$g^k(x)=[1:x_1+ka_1:\cdots : x_n+ka_n].$$ This shows that  all  $g^k(x)$ are on the unique line $L_{p,x}$ passing through $p$ and $x$. Since  the $g^k(x)$ are pairwise different, this implies that $L_{p,x}$ has infinitely many intersection  points with $X$. Therefore, $L_{p,x} \subseteq X$.

Since $X$ is irreducible, every point $y\in X\cap H$ is a limit of points in $X\backslash H$. Hence by continuity, for each point $x\in X\backslash \{p\}$, the line  $L_{p,x}$ is  contained in $X$. This implies that $X$ is a cone with vertex $p$. We obtain a contradiction.
\end{proof}

Now we will look at the case when the special eigenvalue is different from $1$.

\begin{lemma}
\label{lem:intersection:larger:contained}
Assume that $n\geq 2$. Let $X\subseteq \p^n$ be a subvariety. Then there is a number $d(X)$ such that if a line $L$ intersects $X$ at more than $d(X)$ points, then $L\subseteq X$.
\end{lemma}

\begin{proof}
Assume that $X$ is defined as the common zero locus of homogeneous polynomials $P_1,...,P_k$. Let $d(X)$ be the maximal degree of them. Assume that a line $L$ intersects $X$ at more than $d(X)$ points, then $L$ intersects the zero locus of each $P_i$ at more than $d(X)$ points. By degree assumption, this shows that $L$ is contained in the zero locus of each $P_i$. Hence $L\subseteq X$.
\end{proof}

\begin{lemma}
\label{lem:restriction:kernel:eigenvalue:root}
Assume that $n\geq 2$. Let $X\subseteq \p^n$ be an irreducible subvariety which is not a cone. Let $H$ be a hyperplane in $\p^n$ such that $X\not\subseteq H$. Let $g\in \mathrm{Ker}\, (r_H) \cap G_n^X$.  Then the special  eigenvalue $\lambda$ of $g$ is a root of unity. Moreover, its order is bounded by the number  $d(X)$ from above.
\end{lemma}

\begin{proof}
We may assume that  $\lambda$  is different from $1$. Then $g$ is diagonalizable in this case. We may choose homogeneous coordinates $[x_0:x_1:\cdots :x_n]$ of $\p^n$ such that $H$ is defined by $x_0=0$ and that $$g([x_0:x_1:\cdots :x_n])=[\lambda x_0:x_1:\cdots :x_n].$$

Assume by contradiction that the order of $\lambda$ is greater than $d(X)$ (by convention, if $\lambda$ is not a root of unity, then its order is $+\infty$). Let $p$ be the point with homogeneous coordinates $[1:0:\cdots : 0]$. For any point $x \in X\backslash (H\cup \{p\})$ with homogeneous coordinates $[1:x_1:\cdots :x_n]$,   the point  $g^k(x)$ has coordinates $$g^k(x)=[\lambda^k: x_1:\cdots : x_n].$$ This shows that  all of the $g^k(x)$ are on the unique line $L_{p,x}$ passing through $p$ and $x$. Moreover, we note that the cardinality of $$\{g^k(x)\ | \ k\in \mathbb{Z}\}$$ is exactly the order of $\lambda$. By Lemma \ref{lem:intersection:larger:contained}, we obtain that the line $L_{p,x}$ is contained in $X$.  By the same continuity argument as in the proof of Lemma \ref{lem:restriction:kernel:eigenvalue:1:identity}, this implies that for any point $x\in X\backslash\{p\}$, the line $L_{p,x}$ is contained in $X$. Hence $X$ is a cone, and we obtain a contradiction.
\end{proof}

\begin{lemma}
\label{lem:restriction:kernel:finite:set}
Assume that $n\geq 2$. Let $X\subseteq \p^n$ be an irreducible subvariety which is not a cone. Let $H$ be a hyperplane in $\p^n$ such that $X\not\subseteq H$. Then  $\mathrm{Ker}\, (r_H) \cap G_n^X$ is a finite set. As a consequence, we have  $$\dim\, G_n^X \leq \dim G_{n-1}^{X\cap H} + n.$$
\end{lemma}

\begin{proof}
By Lemma \ref{lem:restriction:kernel:eigenvalue:1:identity}, the map $\chi_H: \mathrm{Ker}\, r_H  \cap G_n^X  \to \mathbb{C}^*$ is injective, while Lemma \ref{lem:restriction:kernel:eigenvalue:root} implies that its image has bounded order, hence $\mathrm{Ker}\, (r_H) \cap G_n^X$ is finite.
By Lemma \ref{lem:fixing:hyperplane:non:trivial}, we obtain that  $$\dim\, G_n^X \leq \dim G_{n-1}^{X\cap H} + n.$$

\end{proof}

Now we can conclude Theorem \ref{thm:auto}.

\begin{proof}[{Proof of Theorem \ref{thm:auto}}]

By Lemma \ref{lem:nondegenerate:cutout}, we can repeatedly apply Lemma \ref{lem:restriction:kernel:finite:set} to get
$$
\dim G_n^X \leq \dim G_{n-1}^{X\cap H} + n \leq \cdots \leq \dim G_{c+1}^{X \cap H^{n-c-1}} + n+(n-1)+\cdots + (c+2)
$$

Let $C:=X \cap H^{n-c-1} \subset \BP^{c+1}$ be a general curve section of $X$.  As $C \subset \BP^{c+1}$ is non-degenerate by Lemma \ref{lem:nondegenerate:cutout}, we have
an inclusion $G_{c+1}^{C} \subset {\rm Aut}(C)$, while the latter has dimension at most 3. This gives that
$$\dim G_n^X \leq 3+ \sum_{j=c+2}^n j = \frac{n(n+1)}{2} - \frac{(c-1)(c+4)}{2}, $$
which proves (a).

For (b), if $\dim G_n^X = \frac{n(n+1)}{2}$, then $c=1$ by (a), i.e. $X \subset \BP^n$ is a hypersurface.  As $X$ is not a cone, it must be smooth if it is quadratic. Therefore it remains to show that  if $X$ is a hypersurface of degree at least $3$, then  $\dim G_n^X \leq  \frac{n(n+1)}{2}-1$.
We prove it by induction on the dimension of $X.$
When $\dim X=1$, pick a general line $H$ in $\p^2$. By Lemma \ref{lem:restriction:kernel:finite:set} , we have
 $$\dim G_2^X \leq \dim G_{1}^{X\cap H} + 2.$$
In this case $X\cap H$ is a set of $\deg X \geq 3$ points,  
hence  $\dim G_{1}^{X\cap H}=0$, it follows that $\dim G_2^X \leq 2=\frac{2(2+1)}{2}-1.$
When  $\dim X=n-1\geq 2$, pick a general hyperplane $H$ in $\p^n$. By Lemma \ref{lem:restriction:kernel:finite:set} , we have
 $$\dim G_n^X \leq \dim G_{n-1}^{X\cap H} + n.$$
Then by induction hypothesis we have
$$\dim G_n^X \leq \dim G_{n-1}^{X\cap H} + n \leq  \frac{n(n-1)}{2}-1+n=\frac{n(n+1)}{2}-1.$$

For (c),  if  we assume  further that $X$ is a  smooth hypersurface  of degree greater than $2$, then $\dim G_n^X =0$ (see for example Theorem 1.2 \cite{Poo}).  As a consequence, if  $X$ is smooth and is not a quadratic hypersurface, then $\dim G_n^X \leq \frac{n(n+1)}{2} -3.$
This completes the proof of the theorem.
\end{proof}

\proof[Proof of Theorem \ref{thm:autocone}]
If $C_X=\emptyset$, then we conclude the proof by Theorem \ref{thm:auto}.
Now assume that $C_X\neq \emptyset$.
By Proposition 1.3.3 \cite{Rus},  it  is a linear subspace. For simplicity, we set $r=r_X=\dim C_X$.

Pick a coordinates system of $\p^n$ such that $C_X$ is defined by $x_{0}=\dots=x_{n-r-1}=0.$  Let  $V$ be the subspace of $\p^n$ defined by $x_{n-r}=\cdots =x_{n}=0$ and we identify it with $\p^{n-r-1}$. We let $\pi:\p^n\setminus C_X \to V \cong \p^{n-r-1}$ be the projection
$$ \pi: [x_0:\dots:x_n] \mapsto [x_0:\dots:x_{n-r-1}:0:\cdots :0].$$
Denote the image $\pi(X\setminus C_X)$ by $Y$. Then we have $X\setminus C_X=\pi^{-1}(Y),$ and $Y=X \cap V$. Moreover, $Y$ is not a cone and it is non-degenerate in $\p^{n-r-1}$.
Since $C_X$ is preserved by $G_n^X$,  we see that  $G_n^X \subseteq G_n^{C_X}$.

Each element $g$ in $G_n^{C_X}$ can be represented by a matrix of the shape
$$\left(
\begin{array}{c|c}
 A & {0} \\
 \hline
 B& C
\end{array}
\right)
$$
such that $A$ and $C$ are square matrices  of  dimension  $n-r$  and $r+1$ respectively. We may now define an action of $G_n^{C_X}$ on $V$ as follows.  For each $$y = [a_0:\cdots : a_{n-r-1}:0: \cdots :0] \in V,$$  the new action $g*y$ of $g$ on $y$ is defined  as
$$g*y= \pi (g.y),$$
where $g.y$ represents the standard action of $G_n$ on $\p^n$.
With  the representative above, this action is just defined as $$g*[a_0:\cdots : a_{n-r-1}:0: \cdots :0] = [b_0:\cdots : b_{n-r-1}:0:\cdots :0],$$ where
$$\left( \begin{matrix} b_0\\ b_1\\ \vdots \\ b_{n-r-1} \end{matrix}  \right)  = A \left( \begin{matrix} a_0\\ a_1\\ \vdots \\ a_{n-r-1} \end{matrix}  \right).$$
Thanks to this action, and by identifying $V$ with $\p^{n-r-1}$, we obtain a group morphism $\rho:G_n^{C_X} \to G_{n-r-1}$.

On the one hand, we note that an element $g\in G_n^{C_X}$ belongs to $G_n^X$ if and only if $\rho(g) \in G_{n-r-1}^{Y}$.
On the other hand, an element $g\in G_n^{C_X}$ belongs to $\Ker \rho$ if and only if it can be represented by a matrix of the shape
$$\left(
\begin{array}{c|c}
\rm Id & {0} \\
 \hline
 B& C
\end{array}
\right)
$$
Hence $\dim \Ker \rho = (r+1)(n+1)$. As we can see that $\Ker \rho \subseteq G_n^X$, we obtain that
$$\dim G_n^X = \dim G_{n-r-1}^Y + (r+1)(n+1).$$
Finally,  by Theorem \ref{thm:auto}, we get $$\dim G_n^X \leq \frac{(n-r-1)(n-r)}{2} - \frac{(c-1)(c+4)}{2}+(r+1)(n+1).$$
\endproof

\section{Proof of the main theorem}

\begin{definition}\label{d.VMRT}
Let $X$ be a uniruled projective manifold. An irreducible
component $\sK$ of the space of rational curves
on $X$ is called {\em a minimal rational component} if the
subscheme $\sK_x$ of $\sK$ parameterizing curves passing through
a general point $x \in X$ is non-empty and proper. Curves
parameterized by $\sK$ will be called {\em minimal rational
curves}. Let $\rho: \sU \to \sK$ be the universal family and $\mu:
\sU \to X$ the evaluation map. The tangent map $\tau: \sU
\dasharrow \BP T(X)$ is defined by $\tau(u) = [T_{\mu(u)}
(\mu(\rho^{-1} \rho(u)))] \in \BP T_{\mu(u)}(X)$. The closure $\sC
\subset \BP T(X)$ of its image is the {\em VMRT-structure} on $X$. The
natural projection $\sC \to X$ is a proper surjective morphism and
a general fiber $\sC_x \subset \BP T_x(X)$ is called the VMRT  at the
point $x \in X$.
The VMRT-structure $\sC$ is  {\em locally
flat} if for a general $x \in X$,  there exists  an analytical open subset $U$  of $X$ containing $x$ with an open immersion
  $\phi: U \to
 \C^n, n= \dim X,$ and a projective subvariety $Y \subset \BP^{n-1}$
with $\dim Y = \dim \sC_x$ such that $\phi_*: \BP T(U) \to \BP
T(\C^n)$ maps $\sC|_{U}$ into the trivial fiber subbundle $\C^n
\times Y$ of the trivial projective bundle $\BP T(\C^n) = \C^n \times \BP^{n-1}.$
\end{definition}

\begin{examples} \label{e.IHSS}
An irreducible Hermitian symmetric space of compact type is a
homogeneous space $M= G/P$ with a simple Lie group $G$ and a
maximal parabolic subgroup $P$ such that the isotropy
representation of $P$ on $T_x(M)$ at a base point $x \in M$ is
irreducible.
 The highest weight orbit of the isotropy action on $\BP T_x(M)$
is exactly the VMRT at $x$.  The following table (e.g. Section 3.1 \cite{FH1}) collects basic information on these varieties.

\begin{center}
\begin{tabular}{|c| c| c| c| c| c| }
\hline Type & I.H.S.S. $M$   &  VMRT  $S$ &  $S \subset \BP T_x(M)$  & $\dim \aut(M)$  & $\dim \aut(S)$ \\
\hline I    &  $ \Gr(a, a+b) $ & $\pit^{a-1} \times \pit^{b-1}$ & Segre & $(a+b)^2-1$ & $a^2+b^2-2$ \\
\hline II  & $\mathbb{S}_{r}$ & $\Gr(2, r)$  & Pl\"ucker & $r(2r-1)$ & $r^2-1$  \\
\hline III & $ \Lag(2r)$ & $\pit^{r-1}$ & Veronese  & $r(2r+1)$ & $r^2-1$ \\
\hline IV & $\Q^r$ & $\Q^{r-2}$ & Hyperquadric  &  $(r+1)(r+2)/2$ & $(r-1)r/2$ \\
\hline V & $\mathbb{O}\pit^2$ & $\mathbb{S}_{5}$ & Spinor  &78  & 45  \\
\hline VI & $E_7/(E_6 \times U(1)) $ & $\mathbb{O}\pit^2$ & Severi & 133   & 78 \\
\hline
\end{tabular}
\end{center}

\end{examples}

\begin{lemma}\label{l.IHSS}
Let $M$ be an IHSS of dimension $n$ different from projective spaces and $S \subset \BP^{n-1}$ its VMRT at a general point.
Then
\begin{itemize}
\item[(1)] $\dim \aut(M) \leq \frac{n(n+1)}{2}$  unless $S \subset \BP^{n-1}$ is projectively equivalent to
the Segre embedding of $\BP^1 \times \BP^2$ or the natural embedding of $\Q^{n-2} \subset \BP^{n-1}$.
\item[(2)] The equality holds if and only if $S \subset \BP^{n-1}$ is projectively equivalent to the second Veronese embedding of $\BP^2$.
\end{itemize}
\end{lemma}

\begin{proof}
For Type (I), we have $M={\rm Gr}(a, a+b)$,  $n=\dim M=ab$ and $\dim \aut(M) = (a+b)^2-1$. As $M$ is not a projective space,  we may assume  $b \geq a \geq 2$.
Then the inequality $(a+b)^2-1 \geq \frac{ab(ab+1)}{2}$ is equivalent to $3ab+2 \geq  (a^2-2)(b^2-2)$, which holds if and only if $(a, b)=(2, 2)$ or $(2, 3)$.
In both cases, the inequality is strict.

For Type (II), we have  $M=\BS_r$,  $n=\dim M= r(r-1)/2$ and $\dim \aut(M) = r(2r-1)$.  We may assume $r \geq 5$ as $\BS_4 \simeq \Q^6$. Then one checks that $\dim \aut(M) < \frac{n(n+1)}{2}$.

For Type (III), we have  $M=\rm{Lag}(2r)$,  $n=\dim M= r(r+1)/2$ and $\dim \aut(M) = r(2r+1)$. We may assume $r \geq 3$ as  $\rm{Lag}(4) \simeq \Q^3$. Then one checks that $\dim \aut(M) \leq \frac{n(n+1)}{2}$, with equality if and only if $r=3$. In this case, $S \subset \BP^{5}$ is the second Veronese embedding of $\BP^2$.

For type (IV), we have $M=\Q^r$ and $\dim \aut(M) = (r+1)(r+2)/2$, which does not satisfy $\dim \aut(M) \leq \frac{r(r+1)}{2}$.

For types (V) and (VI),  it is obvious that  $\dim \aut(M) \leq \frac{n(n+1)}{2}$.
\end{proof}

\begin{definition} \label{d.prolong}
Let $V$ be  a complex vector space and $\fg \subset {\rm End}(V)$
a Lie subalgebra. The {\em $k$-th prolongation} (denoted by
$\fg^{(k)}$) of $\fg$ is the space of symmetric multi-linear
homomorphisms $A: \Sym^{k+1}V \to V$ such that for any fixed $v_1,
\cdots, v_k \in V$, the endomorphism $A_{v_1, \ldots, v_k}: V \to
V$ defined by $$v\in V \mapsto A_{v_1, \ldots, v_k, v} := A(v,
v_1, \cdots, v_k) \in V$$ is in $\fg$. In other words, $\fg^{(k)}
= \Hom(\Sym^{k+1}V, V) \cap \Hom(\Sym^kV, \fg)$.
\end{definition}

It is shown in \cite{HM} that for a smooth non-degenerate variety $C \subsetneq \BP^{n-1}$, the second prolongation satisfies $\aut(\hat{C})^{(2)}=0$.

\begin{examples} \label{e.SymGr}
Fix two integers $k \geq 2, m\geq 1$.
Let $\Sigma$ be an $(m+2k)$-dimensional vector space endowed with a
skew-symmetric 2-form $\omega$ of maximal rank. The symplectic Grassmannian
$M=\Gr_\omega(k, \Sigma)$ is the variety of all $k$-dimensional
isotropic subspaces of $\Sigma$, which is not homogeneous if $m$ is odd.
  Let $W$ and $Q$ be vector spaces of
dimensions $k \geq 2$ and $m$ respectively.
Let $\mathbf{t}$ be the tautological line
bundle over $\pit W$.
The VMRT $\mathcal{C}_x \subset \pit T_x(M)$ of $\Gr_\omega(k, \Sigma)$
at a general point is isomorphic to the projective
bundle $\pit((Q \otimes \mathbf{t}) \oplus \mathbf{t}^{\otimes
2})$ over $\pit W$ with the projective embedding given by
the complete linear system $$H^0(\pit W, (Q \otimes \mathbf{t}^*)
\oplus (\mathbf{t}^*)^{\otimes 2}) = (W \otimes Q)^* \oplus \SYM^2
W^*.$$
By Proposition 3.8 \cite{FH1}, we have $\aut(\hat{\mathcal{C}}_x) \simeq (W^* \otimes Q) \sd (\mathfrak{gl}(W) \oplus \mathfrak{gl}(Q))$ and $\aut(\hat{\mathcal{C}}_x)^{(1)} \simeq {\rm Sym}^2 W^*$.
This gives that  $$\dim \aut(\hat{\mathcal{C}}_x) = m^2+k^2+km \mbox{ and } \dim \aut(\hat{\mathcal{C}}_x)^{(1)} = k(k+1)/2.$$
\end{examples}

We denote by $\aut(\mathcal{C}, x)$ the Lie algebra of infinitesimal automorphisms of $\mathcal{C}$, which consists of germs of vector fields whose local flow preserves $\mathcal{C}$ near $x$.
Note that the action of ${\rm Aut}^0(X)$ on $X$ sends minimal rational curves to minimal rational curves, hence it preserves the VMRT structure, which gives a natural inclusion
$\aut(X) \subset \aut(\mathcal{C}, x)$ for $x\in X$ general.

The following result is a combination of Propositions 5.10, 5.12, 5.14 and 6.13 in \cite{FH1}.

\begin{proposition} \label{p.prolong}
Let $X$ be an $n$-dimensional smooth Fano variety of Picard number one. Assume that the  VMRT   $\mathcal{C}_x$ at a general point $x\in X$ is smooth irreducible and  non-degenerate. Then
$$
\dim \mathfrak{aut}(X) \leq n + \dim \mathfrak{aut}(\hat{\mathcal{C}_x})+\dim \mathfrak{aut}(\hat{\mathcal{C}_x})^{(1)}.
$$
The equality holds if and only if the VMRT structure $\mathcal{C}$ is locally flat,  or equivalently if and only if $X$ is an equivariant compactification of $\C^n$.
\end{proposition}

We recall the following result from Theorem 7.5 \cite{FH2}.
\begin{theorem}\label{t.FH}
Let $S \subsetneq \BP V$ be an irreducible smooth non-degenerate variety such that $\aut(\widehat{S})^{(1)} \neq 0$. Then
  $S \subset \BP V$ is projectively equivalent to one in the following list.
\begin{itemize} \item[(1)] The VMRT of an irreducible Hermitian symmetric space of compact type  of rank $\geq 2$.
\item[(2)] The VMRT of a symplectic Grassmannian.
\item[(3)] A smooth linear section of $\Gr(2, 5) \subset \BP^9$ of codimension $\leq 2$.
\item[(4)] A $\BP^4$-general linear section of $\mathbb{S}_5 \subset \BP^{15}$ of codimension $\leq 3$.
\item[(5)] Biregular projections of (1) and (2) with nonzero prolongation, which are completely described in Section 4 of \cite{FH1}.
\end{itemize}  \end{theorem}

\begin{proposition} \label{p.main}
Let $S \subsetneq \BP V$ be an irreducible smooth non-degenerate variety such that $\aut(\widehat{S})^{(1)} \neq 0$. Let $n=\dim V$.
Then
\begin{itemize}
\item[(a)] we have $\dim \mathfrak{aut}(\hat{S}) + \dim \mathfrak{aut}(\hat{S})^{(1)} \leq \frac{n(n-1)}{2}$  unless $S \subset \BP V$ is projectively equivalent to
the Segre embedding of $\BP^1 \times \BP^2$ or the natural embedding of $\Q^{n-2} \subset \BP^{n-1} (n\geq 3)$ .
\item[(b)] The equality holds if and only if $S \subset \BP V$ is projectively equivalent to the second Veronese embedding of $\BP^2$ or a  general hyperplane section of the Segre embedding of $\BP^1 \times \BP^2$.
\end{itemize}

\end{proposition}
\begin{proof}

Consider case (1) in Theorem \ref{t.FH}. Let $M$ be an IHSS and $S \subset \BP V$ its VMRT at a general point. As the VMRT structure is locally flat, we have
$\dim \aut(\hat{S}) + \dim \aut(\hat{S})^{(1)} = \dim \aut(M) - \dim M = \dim \aut(M) - n$ by Proposition \ref{p.prolong}.  Now the claim follows from Lemma \ref{l.IHSS}.

Consider case (2) in Theorem \ref{t.FH}. By Example \ref{e.SymGr}, we have $$\dim \aut(\hat{S}) + \dim \aut(\hat{S})^{(1)} = m^2+k^2+km +  k(k+1)/2$$ with $k \geq 2$ and $n = km+  k(k+1)/2$.
Note that $n \geq km+3$. Assume first that $m \geq 2$, then  we have $m^2+k^2 < (km+3)km/2 \leq n(n-3)/2$, which gives the claim.
 Now assume $m=1$, then it is easy to check that
$1+k^2 \leq n(n-3)/2$ with equality if and only if $(k, m)= (2,1)$.  By Lemma 3.6 \cite{FH1}, this implies that $S \subset \BP V$ is projectively equivalent to a general hyperplane section of the Segre embedding of $\BP^1 \times \BP^2$.

Consider case (3) in Theorem \ref{t.FH}.  If $S$ is the hyperplane section of ${\rm Gr}(2, 5)$, then we have $\dim \aut(\hat{S}) = 16$ and $\dim \aut(\hat{S})^{(1)} = 5$ by Propositions 3.11 and  3.12 in  \cite{FH1}.
 Now assume that $S$ is a codimension 2 linear section of ${\rm Gr}(2, 5)$, then $\dim \aut(\hat{S}) = 9$ and $\dim \aut(\hat{S})^{(1)} = 1$ by Lemma 4.6 \cite{BFM}.  The claim follows immediately.

 Consider case (4) in Theorem \ref{t.FH}. If $S$ is the hyperplane section of $\BS_5$, then $\dim \aut(\hat{S}) = 31$ and $\dim \aut(\hat{S})^{(1)} = 7$ by Propositions 3.9 and  3.10 in  \cite{FH1}.
Now assume that $S_k$ is a  $\BP^4$-general linear section of $\mathbb{S}_5$ of codimension $k=2, 3$. By Propositions 4.7 and 4.11 \cite{BFM}, we have  $\dim \aut(\hat{S_2}) = 19$ and $\dim \aut(\hat{S_3}) = 12$.
By Theorem 1.1.3 \cite{HM}, we have $\dim \aut(\hat{S})^{(1)} \leq \dim V^*$, hence  $\dim \aut(\hat{S_2})^{(1)}  \leq 14$.  By a similar argument as Lemma 4.6 \cite{BFM}, we have $\dim \aut(\hat{S_3})^{(1)} =1$. Now the claim follows immediately.

The case (5) follows from Proposition \ref{p.projection}.
\end{proof}

\begin{proposition} \label{p.projection}
Let $S \subsetneq \BP V$ be an irreducible linearly-normal non-degenerate smooth variety such that $\aut(\widehat{S})^{(1)} \neq 0$. Let $L \subset \BP V$ be a linear subspace such that
the linear projection $p_L: \BP V \dasharrow \BP (V/L)$ maps $S$ isomorphically to $p_L(S)$. Assume that $\aut(p_L(S))^{(1)} \neq 0$. Then
$\dim \mathfrak{aut}(\widehat{p_L(S)}) + \dim \mathfrak{aut}(\widehat{p_L(S)})^{(1)} < \frac{\ell (\ell -1)}{2}$, where $\ell =\dim (V/L)$.
\end{proposition}
\begin{proof}
By \cite{FH1} (Section 4),  $S \subset \BP V$ is one of the followings: VMRT of IHSS of type (I), (II), (III) or VMRT of the symplectic Grassmanninans. We will do a case-by-case check based on the computations in \cite{FH1} (Section 4).

Consider the case of VMRT of IHSS of type (I).  Let $A, B$ be two vector spaces of dimension $a \geq 2$ (resp. $b \geq 2$). Then $S \simeq \BP A^* \times \BP B \subset \BP {\rm Hom}(A, B).$
For a linear subspace $L \subset  {\rm Hom}(A, B)$, we define ${\rm Ker}(L) = \cap_{\phi \in L} {\rm Ker}(\phi)$ and ${\rm Im}(L) \subset B$ the linear span of  $\cup_{\phi \in L} {\rm Im}(\phi)$.
By Proposition 4.10 \cite{FH1}, we have $\aut(\widehat{p_L(S)})^{(1)} \simeq {\rm Hom}(B/{\rm Im}(L), {\rm Ker}(L))$.  As $p_L$ is an isomorphism from $S$ to $p_L(S)$, $\BP L$ is disjoint from ${\rm Sec}(S)$, hence elements in $L$ have rank at least 3.  This implies that $\dim \aut(\widehat{p_L(S)})^{(1)} \leq (b-3)(a-3)$ and
$a, b \geq 4$.  As $\mathfrak{aut}(\widehat{p_L(S)}) \subset \aut(\hat{S}) = \mathfrak{gl}(A^*) \oplus \mathfrak{gl}(B)$, we obtain
$$
\dim \mathfrak{aut}(\widehat{p_L(S)}) + \dim \mathfrak{aut}(\widehat{p_L(S)})^{(1)}  < a^2+b^2+(b-3)(a-3).
$$
Note that $\ell=\dim {\rm Hom}(A, B) - \dim L \geq 3(a+b)-9$ by Proposition 4.10 \cite{FH1}.
Now it is straightforward to check that $a^2+b^2+(b-3)(a-3) < \ell (\ell -1)/2$ since $a+b \geq 8$.

Consider the case of VMRT of IHSS of type (II), then $S={\rm Gr}(2, W)$, where $W$ is a vector space of dimension $r \geq 6$. Let $L \subset \wedge^2 W$ be a linear subspace. We denote by ${\rm Im}(L)$ the linear span of all $\cup_{\phi \in L} {\rm Im}(\phi)$, where $\phi \in L$ is regarded as an element in ${\rm Hom}(W^*, W)$.  By Proposition 4.11 \cite{FH1}, we have $\aut(\widehat{p_L(S)})^{(1)} \simeq \wedge^2(W/{\rm Im}(L))^*$.   As $p_L$ is biregular on $S$, the rank of an element in $L$ is at least 5, hence $\dim \aut(\widehat{p_L(S)})^{(1)} \leq (r-5)(r-6)/2$. On the other hand, we have $\ell = \dim \wedge^2 W - \dim L \geq 6r -11$ by Proposition 4.11 \cite{FH1}.  As $\dim \mathfrak{aut}(\widehat{p_L(S)}) < r^2$, we check easily that $r^2+(r-5)(r-6)/2 < \ell (\ell -1)/2$.

Consider the case of VMRT of IHSS of type (III), then $S = \BP W$, where $W$ is a vector space of dimension $r \geq 4$.  By Proposition 4.12 \cite{FH1}, we have $\aut(\widehat{p_L(S)})^{(1)} \simeq {\rm Sym}^2(W/{\rm Im}(L))^*$, which has dimension at most $(r-2)(r-3)/2$.  On the other hand, we have $\ell \geq 3r-3$ by {\em loc. cit.}, hence we have $r^2+ (r-2)(r-3)/2 < \ell (\ell -1)/2$.

Now consider the VMRT of symplectic Grassmannians. We use the notations in Example \ref{e.SymGr} with  $V=(W \otimes Q)^* \oplus \SYM^2 W^*$, where $\dim W=k \geq 2$ and $\dim Q = m$. By Lemma 4.19 \cite{FH1}, we have $k \geq 3$, hence $k+m \geq 4$.
Let  $L \subset V$ be a linear subspace, then by Proposition 4.18 \cite{FH1}, we have  $\ell \geq 3(k+m)-3$ and  $\aut(\widehat{p_L(S)})^{(1)} \simeq \SYM^2(W/{\rm Im}_W(L))^*$ which has dimension at most $k(k+1)/2$.  As $\dim \aut(\hat{S}) = k^2+m^2+km$, we have
$$
\dim \mathfrak{aut}(\widehat{p_L(S)}) + \dim \mathfrak{aut}(\widehat{p_L(S)})^{(1)}  < k^2+m^2+km+k(k+1)/2 .
$$
Put $s=k+m$, then $k^2+m^2+km+k(k+1)/2 = s^2-km + k(k+1)/2 < 3s^2/2 - km < \ell(\ell-1)/2$ since $k+m \geq 4$.
\end{proof}

Now we can complete the proof of Theorem \ref{t.main}.

\begin{proof}[{Proof of Theorem \ref{t.main}}]
By Proposition \ref{p.prolong}, we have
$$
\dim \mathfrak{aut}(X) \leq n + \dim \mathfrak{aut}(\hat{\mathcal{C}_x})+\dim \mathfrak{aut}(\hat{\mathcal{C}_x})^{(1)}.
$$

If $\mathfrak{aut}(\hat{\mathcal{C}_x})^{(1)}=0$, then $\mathcal{C}_x \subset \BP^{n-1}$ is not a hyperquadric, which implies that
$\dim \mathfrak{aut}(\hat{\mathcal{C}_x}) \leq n(n-1)/2-2$  by Theorem \ref{thm:auto}. This gives that
$$
\dim \mathfrak{aut}(X) \leq n + \dim \mathfrak{aut}(\hat{\mathcal{C}_x}) \leq n + n(n-1)/2-2  < n(n+1)/2.
$$

Now assume $\mathfrak{aut}(\hat{\mathcal{C}_x})^{(1)} \neq 0$.  If $\mathcal{C}_x= \BP T_xX$ for general $x \in X$, then $X$ is isomorphic to $\BP^n$ by \cite{CMSB}.
If $\mathcal{C}_x \subset \BP T_xX$ is a hyperquadric $\Q^{n-2}$, then $X \simeq \Q^n$ by \cite{M}.
If $\mathcal{C}_x \subset \BP T_xX$ is isomorphic to the Segre embedding of $\BP^1 \times \BP^2$, then $X \simeq {\rm Gr}(2, 5)$ by \cite{M}.
Now assume $\mathcal{C}_x \subset \BP T_xX$ is not one of the previous varieties, then by Proposition \ref{p.main} we have
$\dim \mathfrak{aut}(X) \leq n  + n(n-1)/2 = n(n+1)/2$, which proves claim (a) in Theorem \ref{t.main}.

Now assume the equality $\dim \mathfrak{aut}(X) = n(n+1)/2$ holds, then by Proposition \ref{p.prolong}, the VMRT structure is locally flat.  By Proposition \ref{p.main}, $\mathcal{C}_x \subset \BP T_xX$ is either the second Veronese embedding of $\BP^2$ or a general hyperplane section of $\BP^1 \times \BP^2$. Note that these are also the VMRT of ${\rm Lag}(6)$ and a general hyperplane section of ${\rm Gr}(2,5)$, which have locally flat VMRT structure.  This implies that $X$ is isomorphic to ${\rm Lag}(6)$ or a general hyperplane section of ${\rm Gr}(2,5)$
respectively by the Cartan-Fubini extension theorem \cite{HM0}.
\end{proof}

\bibliographystyle{amsalpha}
\bibliography{references}

\bigskip
Baohua Fu (bhfu@math.ac.cn)

MCM, AMSS, Chinese Academy of Sciences, 55 ZhongGuanCun East Road, Beijing, 100190, China
and School of Mathematical Sciences, University of Chinese Academy of Sciences, Beijing, China

\bigskip
Wenhao OU (wenhaoou@math.ucla.edu)

UCLA Mathematics Department, 520 Portola Plaza, Los Angeles, CA 90095, USA

\bigskip
Junyi Xie (junyi.xie@univ-rennes1.fr)

IRMAR, Campus de Beaulieu, b\^atiments 22-23 263 avenue du G\'en\'eral Leclerc, CS 74205
35042 RENNES C\'edex, France

\end{document}